\documentclass{article}[12pts]
\usepackage{hyperref}
\usepackage{amssymb}
\usepackage{mathtools}
\usepackage{mathdots}
\usepackage{amsmath}
\usepackage{amscd}
\usepackage{amsthm}
\usepackage{graphicx}
\usepackage[all]{xy}
\usepackage[dvipsnames]{xcolor}          
\usepackage{epsfig}

\usepackage{lineno}
\usepackage{color}

\newcommand{\C}{\mathbb {C}}

\newcommand{\Addresses}{{% additional braces for segregating \footnotesize
  \bigskip
  \footnotesize

  C.~Cabrera, \textsc{Unidad Cuernavaca del Instituto de Matem\'aticas. UNAM, M\'exico}
  \par\nopagebreak
  \textit{E-mail:}, C.~Cabrera: \texttt{carloscabrerao@im.unam.mx}
  \medskip
  
  P.~Dominguez \textsc{FCFM, BUAP, Mexico}
  \par\nopagebreak
  \textit{E-mail:}, P.~Dominguez: \texttt{pdsoto@fcfm.buap.mx}
  \medskip 

  P.~Makienko  \textsc{Unidad Cuernavaca del Instituto de Matem\'aticas. UNAM, M\'exico}
  \par\nopagebreak
  \textit{E-mail:}, P.~Makienko: \texttt{makienko@im.unam.mx}}}

\newtheorem{theorem}{Theorem}
\newtheorem{lemma}[theorem]{Lemma}
\newtheorem{corollary}[theorem]{Corollary}
\newtheorem{proposition}[theorem]{Proposition}     

\theoremstyle{definition}                             
\newtheorem*{definition}{Definition}
\newtheorem{example}{Example}

\title{On the amenability of semigroups of entire maps and formal power series.}
\author{C. Cabrera,  P. Dom\'inguez  and P. Makienko}

\begin{document}

\maketitle

\footnotetext{This work was partially supported by PAPIIT  IG100523 and CBF2023-2024-1952.}

\begin{abstract}
In this article, we investigate some relations between dynamical and algebraic 
properties of semigroups of entire maps with applications to semigroups of 
formal series. We show that two entire maps fixing the origin 
share the set of preperiodic points, whenever these maps generate a semigroup which 
contains neither free  nor free abelian non-cyclic subsemigroups and 
one of the maps has the origin as a superattracting fixed point. We show
that a subgroup of formal series generated by rational elements is 
amenable, whenever contains no free non-cyclic subsemigroup
generated by rational elements.  We prove that a left-amenable semigroup $S$ 
of entire maps admits a invariant probability measure for a continuous extension of $S$ on
the Stone-C\v{e}ch compactification of the complex plane. Finally, given an entire map $f$, we associate a semigroup 
$S$ such that $f$ admits no ergodic fixed point of the Ruelle operator, whenever 
every finitely generated subsemigroup of $S$ admits a left-amenable Ruelle representation.

\end{abstract}

\section{Introduction}

A theorem due to von Neumann states that amenable groups do not 
contain free non-cyclic subgroups. After this theorem, it has been of big interest  
the classification of groups (semigroups) not containing free 
non-cyclic subgroups (subsemigroups). As an example in the current 
century, we mention in the context of groups a theorem due 
to Margulis (\cite{Margulisfree}) which states: 
\textit{A finitely generated subgroup of the group of orientation
preserving automorphisms of the unit circle $Aut_+(\mathbb{S}^1)$  either admits an 
invariant measure or contains free non-cyclic subgroups}.
Margulis theorem answers a question of Ghys on a version 
of ``Tits alternative'' in this setting, which now it is known as  
the ``Ghys-Margulis alternative". However, $Aut_+(\mathbb{S}^1)$ does 
not satisfy the Tits alternative as it contains an isomorphic copy of the
Thompson group $F,$ see for example \cite{BrinSquier, GhysSergiescu}.

Another result in the spirit of Ghys-Margulis alternative is: 
A finitely generated group of orientation preserving 
automorphisms of the real line either admits an 
invariant Borel measure finite on compact sets or contains 
free non-cyclic subsemigroups (see \cite{Solodov} and also \cite{Beklaryan}).

In contrast, the group of formal series in one variable  is topologically amenable 
and gives an example of  a group satisfying neither the ``Tits alternative" nor 
the ``Ghys-Margulis alternative" as it contains non-cyclic free subgroups 
and admits a finite invariant measure on every Hausdorff compact 
space $X$, where the group acts as a group of self-homeomorphisms
of $X$, see \cite{BabenkoIII}.
 
Recently, in \cite{Tucker, CMAmenability, CMAmenabilityII, 
GhiTucZieve, ZieveZhan} and \cite{PakovichAmenRat}, there
have been established a series of results showing that: 
the Day-von Neumann, the Tits and the Ghys-Margulis alternatives 
all coincide for semigroups of holomorphic endomorphisms 
of the Riemann sphere and relating these algebraic properties 
with dynamical properties. 

Recall that a \textit{non-exceptional rational map} is a non-constant, 
non-injective rational map which is not M\"obius conjugated to either $z^n$, a 
Tchebichev polynomial or a Latt\`es map. Following \cite{CMAmenabilityII}, 
we gather some of these results into the next theorems.

\begin{theorem}\label{th.AmenabilityI}
Let $S$ be a semigroup of rational maps containing a non-exceptional 
element, then the following statements are equivalent.
\begin{enumerate}

\item The semigroup $S$ is right amenable.
\item $S$ contains no non-cyclic free subsemigroups.
\item ${S}$ admits a probability invariant measure $\mu$ 
supported on $\overline{\C}$, which is the measure of 
maximal entropy of every non-injective element of $S$.

\end{enumerate}
\end{theorem}

Let $RIM(S)$ and $LIM(S)$ stand for the sets of right 
and left invariant means for the semigroup $S$, respectively.

\begin{theorem}\label{th.Polynomials}
Let $S$ be a finitely generated semigroup of polynomials with 
generating set $\mathcal{F}.$ If $S$ contains a 
non-exceptional element, then the following statements are equivalent:

\begin{enumerate}
  
  \item For every pair $P$, $Q \in \mathcal{F}$ there exists a point 
  $z_0\in \overline{\C}$ such that $$\#(\{\bigcup_j P^j(z_0) \cap 
  \bigcup_k Q^k(z_0)\})=\infty.$$
  \item For every pair $P,Q\in S$ there are integers $m,n$ such that $P^m=Q^n.$
  \item The semigroup $S$ is amenable with $RIM(S)\subset LIM(S).$
  
 \end{enumerate}
\end{theorem}

The following theorem summarizes the results for polynomial semigroups 
given in \cite{BakerDeMarco, Tucker, CMAmenability, CMAmenabilityII, ZieveZhan, PakovichAmenRat} 
and  \cite{Ye}, which will serve us as a guide. 

\begin{theorem}\label{th.Several}
Let $S$ be a semigroup of non-injective polynomials which 
contains a non-exceptional element, then the following are equivalent:

\begin{enumerate}
 \item Every pair of elements in $S$ shares the Julia set.
 \item Every pair of elements in $S$ shares the measure of maximal entropy.
 \item Every pair of elements in $S$ shares the set of preperiodic points. 
 \item The semigroup $S$ is right amenable.
 \item The semigroup $S$ is nearly abelian. 
 \item The semigroup $S$ contains no non-cyclic free subsemigroups. 
\end{enumerate}
\end{theorem}

This article is devoted to the study of the relations between 
the items of Theorem \ref{th.Several} for semigroups of transcendental entire maps and 
semigroups of formal series, whenever it makes sense.

Also, the motivation for the generalization arises from dynamical systems. 
For instance, every abelian semigroup is nearly abelian and right amenable. Hence, 
an implication from either item 4 or item 5 to item 1  can be regarded 
as generalizations of Baker's conjecture which states: 
\textit{Every pair of permutable entire maps shares the Julia set} 
(\cite{BakerPower} and \cite{BakerConj}). Even more, any equivalence 
to item 1 in Theorem \ref{th.Several} gives 
an algebraic characterization of semigroups of transcendental 
entire maps having the same Julia set. 

Some caution must be taken; outside of the polynomial setting some 
implications are false. In the rational case, there are two rational maps 
$f,g$ sharing the measure of maximal entropy, but not generating a nearly abelian semigroup \cite{Ye}.
Moreover, the concept of measure of maximal  entropy does not make 
sense for transcendental entire maps.  However, according to Theorem \ref{th.Several}, 
the property of sharing the measure of maximal entropy can be replaced by the property of sharing the 
set of preperiodic points. For semigroups of transcendental entire maps, this is a 
stronger property than sharing the Julia sets as the following example shows.

\begin{example}\label{ex.SameJulia} Let $f(z)=z+1-\exp(z)/2$ and 
$g(z)=f(z)+2\pi i$.  Since $f$ commutes with the map $z\mapsto z+2\pi i$, 
then $S=\langle f, g \rangle$  is a non-cyclic  free abelian semigroup satisfying  $J(S)=J(f)=J(g)$. 
But  $Prep(f)\cap Prep(g)=\emptyset$. Moreover, $f$ and $g$ do 
not share a common invariant compact set and, hence, neither a 
common invariant probability measure. Nevertheless, as it will 
be shown in Theorem \ref{th.AMEStoneCech}, the extension of
$S$ as a semigroup of continuous endomorphisms of $\beta(\C)$, 
the Stone-\v{C}ech compactification  of $\C$, admits a finite invariant measure.
\end{example}

More generally, two commuting entire (or rational) maps either 
share an iteration or form a free abelian semigroup  
as in the previous example.  In contrast, according to the main theorem in \cite{EreFunc} 
(see also Proposition \ref{pr.ErLevRitt}), every non-cyclic abelian semigroup $S$ of 
rational  maps  containing an non-exceptional element contains no
non-cyclic free abelian subsemigroup. The simplest example of 
a free abelian non-cyclic semigroup of exceptional rational maps is 
$\langle z^2, z^3 \rangle$.

Recall that two maps $f$ and $g$ satisfy the \textit{dynamical
intersection property} if $$card(\mathcal{O}_f(z_0)\bigcap 
\mathcal{O}_g(z_0))=\infty,$$ for some 
$z_0\in \C$, where $\mathcal{O}_f(z_0)=\bigcup_{n\geq 0} \{f^n(z_0)\}$ is the 
forward orbit of $z_0$ under the map $f$. By Theorem \ref{th.Polynomials}, 
the dynamical intersection property of two polynomials $P$ and $Q$ 
is equivalent to the amenability of the semigroup $S=\langle P, Q \rangle$ 
with additional conditions, such  as $RIM(S) 
\subset LIM(S)$.  However, this is not the 
case for semigroups of transcendental entire maps as Example \ref{ex.SameJulia} shows, 
since $RIM(S)=LIM(S)$ for every abelian semigroup $S$.

The following example presents two maps which generate a free semigroup 
and  satisfy the dynamical intersection property.

\begin{example}  Let $g$ be a polynomial with integer coefficients 
and let $f$ be a transcendental entire map with $f(n)=1$ 
for every $n\in \mathbb{Z}$. For instance, take $f(z)=e^{2\pi i z}$.
If $h= gf$, then $h$ and $g$ share the forward orbit of 
every integer $n\in \mathbb{Z}$. 
\end{example}

From Proposition \ref{pr.Bottcher} (below), if $g(z)=z^m$, 
with $m\geq 2$ and $h$ is as above, then the semigroup $\langle g, h\rangle$ is free. Even more, 
we suggest the following conjecture:
\medskip

\textit{Every  semigroup generated by a non-injective polynomial and a transcendental 
entire map is free.}
\medskip

We sum up other intersection properties for entire maps which will be
discussed throughout article.

\begin{enumerate}
\item The maps $f$ and $g$ have a common periodic point.
\item The maps $f$ and $g$ share the set of all preperiodic points.
\item The maps $f$ and $g$ share a common periodic 
Fatou component. 
\item The maps $f$ and $g$ share a common iteration. 
\end{enumerate}

The simplest non-trivial example of a semigroup satisfying the first property 
is the semigroup of entire maps fixing the origin, which we regard as 
a subsemigroup of the semigroup of formal series of one variable with 
complex coefficients. Let us recall some basic notions and definitions 
of the theory of semigroups of formal series (see for example \cite{BabenkoIV, BakerPower} 
and the references within).   

\medskip

Let $\mathfrak{S}$ be the set of formal power series in one variable $z$ of the form 
$$a_1z+a_2 z^2+..., \textnormal{ where } a_i \in \C,$$
we recall that:

\begin{enumerate}
 
 \item $\mathfrak{S}$ equipped with composition (substitution) is a semigroup.

 \item $\Gamma=\{s\in \mathfrak{S}| a_1\neq 0\}$ is a subgroup of 
 $\mathfrak{S}$, the  group  $\Gamma_1\subset \Gamma$ of all  series with $a_1=1$
 is known as the \textit{group of formal series}. The group  $\Gamma_1$ 
 is normal in $\Gamma$.

 \item Let $s\in \mathfrak{S}$ with $a_1\neq 1$ and let $m$ be 
 the minimum number such that $a_m\neq 0$ in $s$, then there exists
 $\gamma_s\in \Gamma_1$ such that $\gamma_s \circ s \circ \gamma_s^{-1}=a_m z^m$. 
 Moreover, for $m\geq 2$, we can  choose $\gamma_s\in \Gamma$, such that 
 $\gamma_s$ conjugates  $s$ to $z^m$. In analogy with holomorphic maps,  
 we call $\gamma_s$  the (formal) K\"onig's coordinate  when $m=1$ and 
 the (formal) B\"ottcher's coordinate of $s$, when $m\geq 2.$ 
 
 \end{enumerate}

\noindent \textbf{Remark}. The full characterization of conjugacy classes 
(with respect to $\Gamma$) of elements in $\mathfrak{S}$ 
(including the case of $a_1=1$) is given in Proposition \ref{pr.SeriesLineariz} below. 
\medskip

The group $\Gamma_1$ is an object of intensive attention. For example, in \cite{BabenkoIII} it 
is proved that  $\Gamma_1$ is \textit{topologically amenable} as a 
topological group (for definitions and further properties see 
\cite{BabenkoII} and \cite{BabenkoIII}). Topological amenability 
implies, in particular, that if $\Gamma_1$ acts on a compact 
Hausdorff space $X$ by automorphisms, then $X$ admits a 
$\Gamma_1$-invariant probability measure.  However, $\Gamma_1$ 
is not amenable as a group, since it contains non-cyclic free  
subgroups, see \cite{BabenkoIII}.

For $\mathfrak{S}_2=\mathfrak{S}\setminus \Gamma$, the following
result was shown in \cite{ZieveZhan}, see also \cite{PakovichrightamenPow}.

 \begin{proposition}\label{pr.Bottcher}
 Given  $f,g\in \mathfrak{S}_2$ with
 $$f(z)=a_mz^m+...$$ and
 $$g(z)=b_n z^n+...$$
 where $n,m\geq 2$. Let $\gamma_f\in \Gamma_1$ be  the B\"ottcher coordinate of $f$, then
 $$\langle f, g \rangle \simeq \langle z^m,
 \gamma_f \circ g \circ \gamma_f^{-1}\rangle$$ 
 and either 
 \begin{enumerate}
 \item $\langle z^m,\gamma_f \circ g \circ \gamma_f^{-1}\rangle$ is free; or
 \item $ \gamma_f\circ g \circ \gamma_f^{-1}=\omega z^n$ where $\omega$ is a root of unity.
\end{enumerate}
\end{proposition}

This proposition was generalized in \cite{PakovichPowerseries} 
to classify all finitely generated right amenable 
subsemigroups of $\mathfrak{S}_2.$ In general, if 
$f\in \mathfrak{S}\setminus \mathfrak{S}_2$, then  
Proposition \ref{pr.Bottcher} is false, even when $g$ is a finite 
series (polynomial), see \cite{CMAmenabilityII}. 

We organize the results in this paper into the following subsections. 

\subsection{Semigroups of entire maps and formal series without non-cyclic free subsemigroups}

We say that an element $g=a_1z +a_2 z^2+...\in \mathfrak{S}$ 
is \textit{rational} if  $g$ has positive radius of convergence and
converges to a rational function which we also denote by $g$. In  
this setting, the substitution rule on $\mathfrak{S}$ for these rational elements
becomes the composition of rational maps considered as holomorphic 
endomorphisms  of the Riemann sphere. Hence, the whole 
collection $\mathfrak{RS}$ of rational elements 
in $\mathfrak{S}$ forms  a subsemigroup which 
can be identified, by a canonical isomorphism, with the subsemigroup of
rational endomorphisms of $\overline{\C}$ fixing $z=0$. 
If  $\mathfrak{S}_0=\{s\in \mathfrak{S}: s \textnormal{ is a finite series}\}$,  
then $\mathfrak{S}_0 \subset \mathfrak{RS}$.
Notice that both $\mathfrak{S}_0$ and $\mathfrak{RS}$ are dense subsemigroups 
of $\mathfrak{S}$ with respect to the topology of 
convergence on coefficients. 

\begin{definition}

We say that a group (semigroup) $G\subset \mathfrak{S}$ is a  \emph{rational group 
(semigroup)} if $G$ is generated as a group (semigroup) by elements 
in $\mathfrak{RS}$.
\end{definition}

Let $Ent(\C)$ be the semigroup of all entire maps on $\C$. If
$\mathfrak{H}=\{s\in \mathfrak{S}: s(z)=a_1 z+a_2 z^2+...\textnormal{ 
is absolutely convergent } \forall z\in \C \},$ then $\mathfrak{H}$ 
can be identified with the subsemigroup $zEnt(\C)\subset Ent(\C)$ 
of entire maps fixing the origin. Even more, $\mathfrak{S}_0$ is a 
dense subset of $\mathfrak{H}$ in the compact open 
topology of the complex plane, and 
$\mathfrak{S}_0=\mathfrak{H}\cap \mathfrak{RS}$. 
Hence, the semigroup $\langle \mathfrak{H}, \mathfrak{RS}\rangle$  can 
be identified with the semigroup of meromorphic maps fixing the 
origin equipped with the composition as the multiplication. This 
identification induces a non-injective action for a semigroup 
$S\subset \langle \mathfrak{H}, \mathfrak{RS}\rangle$ on  subsets of $\overline{\C}$, 
even when $S$ generates a group in $\mathfrak{S}$. 

Now, we formulate one of the main results in this paper which 
gives a positive answer to a version of the Day-von Neumann 
amenability conjecture for rational subgroups of $\Gamma.$

\begin{theorem}\label{th.polysemi}
A rational subgroup $G\subset \Gamma$ is amenable whenever 
$G$ contains no non-cyclic free rational subsemigroup. 
\end{theorem}

In the exceptional case, the reciprocal of the Theorem  \ref{th.polysemi} is false. Indeed,
the semigroup of affine maps $S=\langle 2z,2z+1\rangle$ is free (see \cite{CMAmenabilityII}),
but the group $G$ generated by $S$ is amenable as a subgroup of $\mathrm{Aff}(\C)$. Thence 
$\tilde{G}=\frac{1}{z}\circ G \circ\frac{1}{z}\subset \Gamma$ is a rational amenable group containing 
a non-cyclic free rational subsemigroup $\tilde{S}=\frac{1}{z}\circ S \circ\frac{1}{z}$. 

\medskip 
\noindent \textbf{Remark.} In the non-exceptional case, the reciprocal of 
Theorem \ref{th.polysemi} is an interesting open question.  
\medskip

The following corollary can be regarded as a version of the classical 
``Tits alternative'' in this setting.

\begin{corollary}\label{cor.titsalternative}
Let $G$ be a group as in Theorem \ref{th.polysemi} with a non-injective rational element, 
then $G$ is either virtually cyclic or abelian. 
\end{corollary}

Applying Theorem \ref{th.AmenabilityI}, we obtain the following 
immediate corollary.

\begin{corollary}\label{cor.polysemi}
A rational group $G\subset \Gamma$ is amenable whenever 
every rational non-cyclic subsemigroup in $G$ is right amenable.
\end{corollary}

It would be interesting to have an analogue of Theorem \ref{th.polysemi} 
for either groups or semigroups generated by elements in $\mathfrak{H}.$

In what follows, we formulate a series of results of dynamical type which 
provides, in particular,  a partial answer to Baker's conjecture.

\begin{theorem}\label{th.superatracting}
 Let $f,g$ be a pair of non-injective entire maps sharing a periodic 
 point $z_0$ which is superattracting for $g$, then $f$ and $g$ share 
 the set of preperiodic points whenever the semigroup $\langle f, g \rangle$
 contains neither free nor free abelian non-cyclic subsemigroups. 
\end{theorem}

An inductive argument immediately gives the following corollary.

\begin{corollary} Let $f$ and $g$ be as in Theorem \ref{th.superatracting}, 
then every $h\in \langle f, g \rangle$ shares the Julia set with $g$.
\end{corollary}

When the common periodic point $z_0$ is not superattracting for one of the maps, we have 
the following theorem. 

\begin{theorem}\label{th.sharingiteration}
Let $f,g \in  \mathfrak{S}$ such that $f\in \Gamma$ is an element of infinite order and $g\neq f^{-n}$ for $n\in \mathbb{N}$.
Then, the following are equivalent: 

\begin{enumerate}
 \item $f$ shares an iteration with $g$.
 \item The semigroup $\langle f, g \rangle$ is 
    abelian, but not free abelian. 
  \item $\langle f, g \rangle$ is abelian and contains  no free abelian non-cyclic subsemigroup. 
    \end{enumerate}
\end{theorem}

The latter theorem is an analogue of Eremenko-Ritt's theorem (Proposition \ref{pr.ErLevRitt} below)  
in the case of formal series. 

Now, we state the version of Theorem \ref{th.superatracting} 
where we consider a periodic component instead of a periodic point. 

\begin{theorem}\label{th.preperiodic}

 Let $f,g$ be a pair of non-injective entire maps  sharing a 
 bounded periodic Fatou component $D$, then $f$ and $g$ share 
 the set of preperiodic points whenever the semigroup $\langle f, g \rangle$
  contains neither free nor free abelian non-cyclic subsemigroups. 

\end{theorem}

As the following corollary shows, when $D$ is an  attracting, but
not superattracting component for an element of $S$, the condition 
of no free  non-cyclic  subsemigroups is enough to  share 
the set of preperiodic points. 

\begin{corollary}\label{cor.boundeddomain}
Let $f$ and $g$ be as in Theorem \ref{th.preperiodic} and assume  
$D$ is an attracting, but not superattracting, Fatou component for $g$. 
If $\langle f,g \rangle$  contains no free non-cyclic 
subsemigroups, then  $f$ 
shares the set of preperiodic points with $g$. 
\end{corollary}
 
Also, we have the following corollary. 

\begin{corollary}\label{cor.SiegelAbel}

Assume that two entire maps $f$ and $g$ share a common bounded invariant Fatou component $D$
which is a Siegel disk for both maps, then $\langle f, g \rangle$ is abelian.
\end{corollary}

By Theorem \ref{th.Several}, two polynomials $f$ and $g$ sharing the set of preperiodic points 
generate a semigroup with no free nor free abelian non-cyclic subsemigroups.  The analogous 
statement for transcendental entire maps remains an open question. Nevertheless, we have the 
following theorem which serves as a reciprocal to Theorem \ref{th.preperiodic}.

\begin{theorem}\label{th.boundedFatou}
Let $f$ and $g$ be entire maps sharing an invariant bounded 
Fatou component $D$. Assume $D$ is an attracting component of $f$ and 
the maps $f$ and $g$ share the set of preperiodic points, then the semigroup 
$\langle f, g \rangle$ contains neither free nor free abelian non-cyclic subsemigroups.
\end{theorem}

By analogy with the results of 
\cite{CMAmenability} and \cite{CMAmenabilityII}, the 
condition of having no free non-cyclic  subsemigroups 
can be replaced with the condition that every non-cyclic subsemigroup is right amenable.

\subsection{Right amenable and nearly abelian semigroups of entire maps}
 
Observe that every subsemigroup  $S \subset Ent(\C)$ is right amenable, whenever $S$ contains a 
constant endomorphism, say  $g(z)=c$. Indeed, the delta measure $\delta_g$ 
based on $g$, is a right invariant measure defined on the algebra of all subsets of $S$.  Furthermore,  if $c$ is a fixed point 
of every element in $S$, then $\delta_g$ is  left invariant and, hence, $S$  is amenable. The affine group $\mathrm{Aff}(\C)$, 
the  group of injective  entire maps, is amenable as it is 
a semidirect product of abelian  groups. However, $\mathrm{Aff}(\C)$ contains free non-cyclic
subsemigroups (see, for example, \cite{CMAmenability,CMAmenabilityII}).

Right amenable polynomial semigroups are reasonably characterized in Theorem \ref{th.Several}.
However, the right amenable semigroups of transcendental entire maps are less known. 
 
Next, we show that the class of right  amenable semigroups of 
non-constant, non-injective, entire maps is sufficiently 
rich.  For this we need a generalization of the notion 
of nearly abelian semigroups which was introduced in
holomorphic dynamics in \cite{HinkMartSemi}.

\begin{definition}
    A semigroup $S$ of non-constant entire maps is called \emph{nearly abelian} if 
    for every pair of elements $f,g \in S$, there exists 
   $\gamma\in \mathrm{Aff}(\C)$ such that
    
    \begin{equation}\label{eq.nearlyabel}
         f\circ g=\gamma \circ g \circ f. \tag{1}
    \end{equation}

    The whole collection $K(S)$ of all such $\gamma$ satisfying Equation 
    \eqref{eq.nearlyabel}, for all pairs $f,g\in S$, is 
    called the \emph{commutator set} of $S$.
\end{definition}

In the latter definition, we do not assume that $S$ consists only of transcendental entire maps. For example, 
$S$ may include affine maps. Also, we neither assume that $K(S)$  forms a compact subfamily  
of $\mathrm{Aff}(\C)$ nor that the elements of $K(S)$ leave the Fatou set of $S$ invariant
(compare with \cite{HinkMartSemi}).  If $S$ consists only of polynomials, the compactness 
of $K(S)$ follows from Lemma \ref{lm.algebraic} (below), theorem in \cite{AtelaHu}  and 
Corollary 24 in \cite{CMAmenability}. The following is an example of a nearly abelian semigroup 
with non-compact commutator set $K(S)$.

\begin{example} 
Let $f_n(z)=\exp(z)+2\pi i n$ for $n\in \mathbb{Z}$, then $\{f_n\}$ generates 
an infinitely generated nearly abelian semigroup satisfying 
the relations:

\begin{itemize}
\item For every $i,j\in Z$, $f_i\circ f_j=f^2_i$ and;
\item The commutator set is an infinite cyclic group generated 
by $z\mapsto z+2\pi i.$ 
\end{itemize}

\end{example}

Next theorem corresponds to the implication from part 5 to part 4 of Theorem \ref{th.Several}.

\begin{theorem}\label{th.nearlyabelian}
Let $S\subset Ent(\C)$ be a nearly abelian semigroup 
with commutator set $K(S)$ and let $G(S)$ be the group 
generated by $K(S)$. Then, the semigroup $T(S)=\langle G(S), 
S \rangle$, generated by $S$ and 
$G(S)$ is right amenable.     
\end{theorem}
The semigroup $T(S)=\langle G(S),S \rangle$, in 
Theorem \ref{th.nearlyabelian} is also  nearly abelian. 
As shown in \cite{CMAmenabilityII}, if $K(S)$ is finite, then $S$ 
is right amenable. We extend this statement for entire maps in 
the following corollary.

\begin{corollary}\label{cor.nearlyabelian}
Let $S$ be a nearly abelian semigroup of non-injective entire maps. Then $S$ is 
right amenable whenever one of the following conditions holds true. 
\begin{enumerate}
    \item The commutator set $K(S)$ is finite.
    \item The semigroup $T(S)$ in Theorem \ref{th.nearlyabelian} is 
 cancellative and $S$ contains no free non-cyclic  subsemigroups.
\end{enumerate}
\end{corollary}

The second item in Corollary \ref{cor.nearlyabelian} is sharp. Indeed, by Proposition \ref{pr.Bottcher} the semigroup 
$S=\langle z^2,\lambda z^2 \rangle$ is free, whenever
$\lambda$ is not a root of unity.  But  $S$ is nearly abelian whose commutator $K(S)$ is infinite and 
generates a group $G(S)$ which is a subgroup of the cyclic group $\langle z\mapsto \lambda z\rangle$. 
Then by Proposition \ref{LemmaKlawe},  $\langle G(S),z^2\rangle$ is a
right amenable semigroup  containing the nearly abelian free semigroup $S$.

\subsection{Left amenability and Stone-C\v{e}ch extension}

As a consequence of von Neumann and Margulis theorems mentioned before, every 
left amenable group of automorphisms of the unit circle admits an 
invariant probability measure. The following theorem 
serves as an analogue of this statement for semigroups of entire maps.

\begin{theorem}\label{th.AMEStoneCech} Let $S\subset Ent(\C)$ be 
a left amenable semigroup of entire maps, then the semigroup
 $\tilde{S}$, consisting of continuous extensions of elements of $S$ 
 onto $\beta(\C)$, the Stone-\v{C}ech compactification of the complex plane,
admits an invariant probability measure on $\beta(\C)$.
\end{theorem}

Example \ref{ex.SameJulia} shows that the support of the 
invariant measure of Theorem \ref{th.AMEStoneCech} may be 
a subset of $\beta(\C)\setminus \C.$

\subsection{Ruelle representantion}

We discuss left amenability for a representation (bounded homomorphism) of a 
semigroup $S$ of entire maps $S$ which we call the \emph{Ruelle representation}.
More precisely, we will consider the push-forward action of $S$ on 
the subspace of integrable measurable quadratic differentials supported on a 
completely $S$-invariant positive Lebesgue measurable set. 

Ruelle representation is  closely related with the quasiconformal deformations of the semigroup $S.$
Propositions \ref{pr.Estrella}  and \ref{prop.tresestrellas}, which present independent interest,
 show that a semigroup $S\subset Ent(\C)$  satisfying the Levin 
relations is quasiconformally deformable whenever $S$ contains 
a quasiconformally deformable generator. 
Even more, the semigroup $S=\langle  \gamma\circ g\rangle$,
with $\gamma\in Deck(g)=\{\gamma\in \mathrm{Aff}(\C): g(\gamma)=g\}$, is structurally 
stable over any analytic family $\mathcal{F}\subset Ent(\C)$ for 
which there exists an element $f\in S$ structurally 
stable in $\mathcal{F}$.

For polynomials, we have the following theorem.

\begin{theorem}\label{th.poldeformable}
A non-exceptional semigroup $S$ of polynomials is quasiconformally 
deformable if and only if $S$ is right amenable and contains a 
quasiconformally deformable generator. 
\end{theorem}

We say that $f \in Ent(\C)$  belongs to the class 
$\mathcal{R}$, whenever there exists a rectifiable ray
$R\subset \C$ containing the singular set of $f$ (that is, the 
union of the sets of critical and asymptotic values).
To every non constant, non-linear function $f \in \mathcal{R}$, we associate $D(f)$ 
the right amenable semigroup of operators acting on 
$L_1(A)$ space for every  completely $f$-invariant set $A$ of positive Lebesgue measure.

Proposition \ref{pr.Ruellerightamenable} shows that the action 
of $D(f)$ on $L_1(A)$ is left amenable, whenever $A$ does 
not admit invariant Beltrami differentials. Conversely,
we can show the following theorem. 

\begin{theorem}\label{th.leftRuelle}
Let $f \in \mathcal{R}$  and  
let $A\subset J(f)$ be a completely  invariant set, 
ergodic with respect to $f$. Then there is no non-zero fixed 
point of the Ruelle  operator in $L_1(A)$, whenever the action of 
every finitely generated  subsemigroup $S\subset D(f)$ on 
$L_1(A)$ is left-amenable. 
\end{theorem}

As it was observed in \cite{EremLyubPathological}, Sullivan's no invariant line fields conjecture is false for generic transcendental entire maps. 
However, Theorem  \ref{th.leftRuelle} suggests the 
following  version of Sullivan's conjecture for entire maps:

 \textit{There are no ergodic $f$-invariant measurable 
 differential of type $(p,q)\in \mathbb{N}^2$ supported on 
the Julia set of an entire map $f$.}
\medskip

Here a measurable $f$-invariant differential $\mu$ of type 
$(p,q)\in \mathbb{N}^2$ is ergodic whenever $f$ is ergodic on the support of $\mu$. 

Section 2 gives a brief account of the necessary notions and theorems.  
The proofs of the theorems will be given in Section 3, which is organized accordingly 
to the subsections of the introduction.

\section{Preliminaries}
 We include some basic notions and results which will be needed for
 the proofs of the theorems given in the introduction. 
 
\subsection{Amenability of semigroups}
For convenience of the reader, we recall some basic notions 
 and facts of the theory of amenability of semigroups 
 based on the works  \cite{DayMeans} and \cite{DayAmenable}.
A semigroup $S$ acts on itself from the 
left and  the right by $l_\gamma(g)=\gamma  g$ and 
$r_\gamma(g)=g \gamma$ for $\gamma\in S$, respectively. Let 
$$\mu:\mathcal{A}_S\rightarrow \mathbb{\C}$$ be a complex valued function defined on 
the algebra $\mathcal{A}_S$ of all subsets of $S$. Then $\mu$  is a 
finitely additive complex valued measure (or \textit{charge}) if it 
satisfies the following conditions:  $\mu(\emptyset)=0$, 
$\mu(A\cup B)=\mu(A)+\mu(B)$ for disjoint sets $A,B\in \mathcal{A}_S$ 
and $\mu$ has finite total variation over $S$.

\begin{definition} 
A semigroup $S$ is called \emph{right amenable}, if it
possesses a right $S$-invariant  probability charge $\mu$, that is $\mu\geq 0$, 
$\mu(S)=1$ and, for every $A\in \mathcal{A}_S$ and $\gamma \in S$, we have 
$$\mu(r_\gamma^{-1}(A))=\mu(A).$$  Analogously, \emph{left amenability} is 
defined in terms of the left action. Finally, $S$ is called 
\emph{amenable} if it is both left and right amenable.
\end{definition}

By Theorem 1 (IV.5.1, page 258)  in \cite{DunfordSchwartz}, the space of 
finite complex valued charges on $S$ is isometrically 
isomorphic to the dual space of the space $L_\infty(S)$ of all bounded
functions on $S$, equipped with the supremum norm. In this 
situation, a non-negative (right, left)-invariant charge $\mu$ 
determines a non-negative (right, left) invariant 
functional $L_\mu$ which is called either a right, left or both
invariant mean. 

If $S$ is a group, then the notions of left and right amenability are equivalent, but for 
semigroups these notions may be different.  Some basic facts about amenable 
semigroups are the following:

\begin{enumerate}
 \item Every finite group is amenable, but not every finite semigroup is amenable.
 \item Every abelian semigroup is amenable.
 \item Every subgroup of an amenable group is amenable.
 \item Every semigroup is isomorphic to a subsemigroup of an amenable semigroup.
 \item Every semigroup $S$ comes with an opposite semigroup: If $*$ is 
 the product in $S$, then $a*_{op}b=b*a$ is the opposite product. In this way,
 left attributes can be converted into right attributes and vice versa. For instance, 
 the opposite semigroup of a left amenable semigroup is right amenable. 
 \item (Day exhaustion theorem) If the semigroup $S=\bigcup S_n$,  where $S_n$ are 
 subsemigroups such that for every $m,n$ there exists $k$ with $S_m \cup S_n \subset S_k$. 
 Then $S$ is right (left) amenable whenever the semigroups $S_n$ are right (left) amenable for every $n$.
\end{enumerate}
 
Following  \cite{CMAmenabilityII} and \cite{DayMeans} (see also \cite{CMAmenability}), we 
introduce a weaker property than amenability, namely that of 
amenability of bounded representations also known as $\rho$-amenability.

A proper right (left) $S$-invariant subspace $X\subset L_\infty(S)$ 
is called either \textit{right} or \textit{left amenable} whenever $X$ 
contains constant functions and there exists a mean $M$ so that the 
restriction of $M$ to $X$ is an invariant functional for either 
the right or left actions of $S$ on $X$, respectively. In other 
words, $X$ is right (left) $S$-invariant and admits a respective 
invariant non-negative state. Notice that every semigroup admits 
an amenable subspace, for example, the subspace  of constant 
functions is always amenable for every semigroup $S$. 

Now, let $\rho$ be a bounded (anti-)homomorphism from $S$ into $End(B)$, 
where $B$ is a Banach space, and $End(B)$ is the semigroup 
of continuous linear endomorphisms of $B$. Let $B^*$ be 
the dual space of $B$. Given a pair $(b,b^*)\in B\times B^*$ 
consider the function $f_{(b,b^*)}\in L_\infty(S)$ given by
$$f_{(b,b^*)}(s)=b^*(\rho(s)(b)).$$  Let $Y_\rho\subset L_\infty(S)$, 
be the closure of the linear span of the family of functions 
$\{ f_{(b, b^*)} \}$ for all pairs $(b,b^*)\in B\times B^*.$ 
Finally let $X_\rho$ be the space generated by $Y_\rho$ and 
the constant functions. Notice that $X_\rho$ and $Y_\rho$ are
both right and left invariant. 

\begin{definition}
We will say that a bounded (anti-)homomorphism $\rho$ is either 
\emph{right amenable or left amenable}, whenever $X_\rho$ is 
either a right or left amenable subspace of $L_\infty(S)$, 
respectively.  Also we say that $S$ is $\rho$ \emph{right amenable} 
or $\rho$ \emph{left amenable}, whenever $\rho$ has the respective property. 
Equivalently, that the $\rho$-action of $S$ on $B$ is either 
right amenable or left amenable, respectively.
\end{definition}

In general, amenability is a stronger property than $\rho$-amenability. Indeed, we  
recall Day's theorem from \cite{DayMeans}.

\begin{theorem}\label{th.Day} A  semigroup $S$ is right amenable 
(left amenable) if and only if $S$ is $\rho$-right amenable 
($\rho$-left amenable) for every bounded (anti)-representation $\rho$.
\end{theorem}

We state the necessary results from the theory of amenable semigroups 
(see \cite{Paradoxal}).

\begin{lemma}\label{lemma.leftright}
A semigroup $S$ is either right, left or both amenable, whenever
every countable subsemigroup of $S$ is either right, left or both amenable, respectively.
\end{lemma}

A semigroup $S$ is left cancellative if for all $a,b,c \in S$, 
the equation  $ca=cb$ implies $a=b$. Analogously, right cancellative
semigroups are defined. A semigroup is cancellative if it is both right 
and left cancellative. For example, every semigroup of non-constant 
entire maps is right cancellative. 

The following theorem is proved in \cite{Donnelly}. 
\begin{theorem}\label{th.Donnelly} Let $S$ be a cancellative semigroup. 
Let $T$ be a subsemigroup of $S$ such that $T$ does not contain 
a free subsemigroup on  two generators. If S is
left amenable, then $T$ is left amenable.
\end{theorem}

Let us recall some results due to Klawe (see \cite{Klawe}) on semidirect 
product of semigroups. If $S,T$ are two semigroups 
with a homomorphism $\psi:S\rightarrow End(T)$, then the 
cartesian product  $T\times S$ possess a semigroup structure, called the 
\textit{semidirect product} $T\rtimes_\psi S$ given by the product $$(t_1,s_1)*(t_2,s_2)=(t_1\psi(s_1)(t_2),
s_1s_2).$$

We summarize some results in \cite{Klawe} in the following proposition.

\begin{proposition}\label{LemmaKlawe}
Let $G$ and $S$ be semigroups and $\psi:S\rightarrow End(G)$ be a homomorphism. 
\begin{itemize}
\item If  $G$ and $S$ are right amenable, then  
$G\rtimes_\psi S$ is right amenable; 

\item If $G$ and $S$ are  amenable and $\psi(s)$ is 
a surjection  for every $s\in S$, then  $G\rtimes_\psi S$ 
is an amenable semigroup.
\end{itemize}
\end{proposition}

\subsection{Ruelle and Koopman representations}

In this subsection, we discuss two important representations (actions)
in dynamics: the pull-back and push-forward actions on the space of continuous functions and
the space of integrable measurable quadratic forms, respectively, these are known as 
the Koopman and  Ruelle representations, respectively. 
First, we consider  the Koopman representation for entire maps. 

By the Krylov-Bogolyubov theorem, any continuous endomorphism $f$ of a compact 
Hausdorff space $X$ admits a finite invariant measure which defines a 
$K_f$-invariant continuous linear functional on $C(X)$, where $K_f$ 
is the Koopman operator (or composition operator) defined by 
$K_f(\phi)=\phi\circ f$. 

When compactness fails, such as in the case of  the complex plane $\C$,  
then a continuous endomorphism of $\C$ may or may not admit a 
finite non-atomic invariant measure on $\C$. For example, the 
map $z\mapsto z+1$ does not admit finite invariant measures on $\C$. 

Therefore, as in the group setting, we define Koopman 
representation as a special subclass 
of $\rho$-representations into the semigroup of linear 
continuous endomorphisms of the space $C(X)$ of continuous complex 
functions over a compact Hausdorff space $X$.

Given a compact Hausdorff space $X$, let $E(X)$ be the semigroup
of continuous endomorphisms of $X$. Then the representation $\rho_X:E(X)\rightarrow End(C(X))$ 
given by  $$\rho_X(f)=K_f,$$ where $K_f(\phi)=\phi(f)$
is called the \emph{Koopman (anti)-representation}.

\begin{definition}
For an abstract semigroup $S$, we will say that $S$ admits a \emph{Koopman
representation} $\rho$ if there exists a compact Hausdorff space $X$
and a monomorphism $i:S\rightarrow E(X)$ such that $\rho=\rho_X\circ i.$
\end{definition}

For semigroups $S$ of entire maps, natural examples of Koopman 
representations are continuous extensions of elements from $S$ to 
a suitable compactification of $\C$. We say that a compactification 
$\tilde{\C}$ is an \textit{admissible compactification} of the complex plane $\C$ 
for the semigroup $S$ whenever $\tilde{\C}$ is a Hausdorff 
compact space and the action $S$ on $\C$ continuously 
extends to a continuous action on $\tilde{\C}$ such that the 
correspondence $f\mapsto \tilde{f}$ is an isomorphism. 
The one point compactification  $\overline{\C}$ of $\C$ is 
an example of an admissible compactification for the semigroup of
polynomials $Pol(\C)$. Another example of an admissible compactification for $Pol(\C)$ is
the compactification of $\C$ by a circle at infinity.

By its universal property, the Stone-\v{C}ech compactification of $\C$
is an admissible compactification for the whole semigroup $Ent(\C)$. 
Perhaps, as a consequence of Picard theorem, $\beta(\C)$ is the unique, up to 
homeomorphism, admissible  compactification for $Ent(\C)$.

Next, we define the  Ruelle representation.

\begin{definition} Let $f$ be an entire map in the class $\mathcal{R}$.
Then $f$ acts on $L_1(\C)$ as a ``push-forward map", more precisely, 
as the linear continuous operator given by the formula  
$$f_*(\phi)(x)=\sum_{y\in f^{-1}(x)} \frac{\phi(y)}{(f')^2(y)},$$  
which is called the \emph{Ruelle transfer operator} or just the \emph{Ruelle operator}. 
\end{definition}

The correspondence  $$f \mapsto f_* $$
defines an isomorphism onto its image, whenever it exists, and  is called the 
\textit{Ruelle representation}. 

The operator $B_f: L_{\infty}(\C) \rightarrow L_{\infty}(\C)$ 
given by $$B_f(\varphi)=\varphi(f) \frac{\overline{f'}}{f'}$$ 
is the pull-back map on Beltrami forms and we call it the 
\textit{Beltrami operator} and is also a linear continuous 
endomorphism of $L_\infty(\C)$ with $\|B_f\|_\infty=1$. 
The set of fixed points $Fix(B_f)\subset L_\infty(\C)$ is called 
the space of invariant Beltrami differentials. By quasiconformal
theory, the open unit ball of $Fix(B_f)$ generates all 
quasiconformal deformations of the map $f.$ 

The relationship between the Ruelle and Beltrami operators in class $\mathcal{R}$ is given by the 
following proposition (see \cite{DomMakSien} and \cite{MakRuelle}).

\begin{proposition}\label{pr.RuelleBeltrami}
Let $f$ be an entire map in class $\mathcal{R}$, then 
 $f_*$ defines a linear continuous endomorphism 
of $L_1(\C)$ with norm $\|f_*\|_1\leq 1$ and $B_f$ 
is the dual operator to  $f_*$.
\end{proposition}

This proposition may hold true in a more  general setting but is not 
clear how to proceed, for instance, in the case when the singular set 
has positive measure or, even, when the singular set is the whole plane. 
Thence, we introduced the class $\mathcal{R}$ by technical reasons. 

Following \cite{KornfeldLin} and \cite{Krengel}, 
we summarize the basic properties of dynamics of 
linear contractions acting on Banach spaces. For a 
linear endomorphism $L$ of a Banach space $B$, the 
\textit{C\'esaro averages} $C_n(L)$ are defined for $x\in B$ by

$$C_n(L)(x)=\frac{1}{n}\sum_{i=0}^{n-1}L^i(x).$$

\begin{definition}
An operator $L$ on a Banach space $B$ is called \emph{mean-ergodic} if
there exists $M>0$ such that
    $$\|L^n\|\leq M<\infty,$$ for every $n\geq 0$, 
and the C\'esaro averages $C_n(L)$ converges in norm for 
every element of $B$.
\end{definition}

The following characterization of mean-ergodicity can be 
found in \cite{Krengel}.

\begin{proposition}[Separation Principle]\label{pr.separationprinc}
An operator $L$ acting on a Banach space $B$ is mean ergodic 
if and only if it satisfies the principle of separation of fixed points:

For every  $x^*\in B^*$ fixed by $L^*$, the dual operator to $L$, 
there exists $y\in B$ a fixed point of $L$ such 
that $(x^*,y)=x^*(y)\neq 0.$
\end{proposition}

An operator $L$ acting on a Banach space $B$ with $\|L\|\leq 1$ 
is called \textit{weakly almost periodic} whenever the sequence
$\{L^n(x)\}$ forms a weakly sequentially precompact 
subset of $B$, for every $x\in B.$

The following theorem is proved in \cite{KornfeldLin}.

\begin{theorem}\label{th.KORNFELIN}
Let $L$ be a positive operator with $\|L\|\leq 1$ acting on $L_1(X,\mu)$ 
space. Then, $L$ is weakly almost periodic if and only if $L$ is mean ergodic. 
\end{theorem}

\subsection{Formal series, relations and functional equations }

Recall that $\mathfrak{S}$ stands for the semigroup of formal series 
in one variable with complex coefficients of the form 
$$a_1z+a_2z^2+a_3z^3+...$$ equipped by substitution as operation
and the topology of convergence on coefficients and $\Gamma$ is  the  
subgroup of all invertible elements in $\mathfrak{S}.$

The following fact is showed in \cite{Scheinberg}. For every $g\in \mathfrak{S}$ 
and $\gamma\in \Gamma$, let $g_\gamma=\gamma \circ g \circ \gamma^{-1}$.

\begin{proposition}\label{pr.SeriesLineariz}
For every $g=a_1z+a_2z^2+...\in \mathfrak{S}$ there exists $\gamma \in \Gamma$ 
such that the following statements are true:

\begin{enumerate}
    \item $g_\gamma=a_1z$ whenever $a_1\neq 0$ and either $a_1$ is not a root of unit 
    or $a_1^k=1$ is a primitive $k$-root with $g^k=id.$
    \item $g_\gamma=z^m$ with $m\geq 2$ and $m$ is the minimal natural number $m$ such that $a_m\neq 0$.
    \item $g_\gamma=a_1z+b_1z^n+c_1z^{2n-1}$, if $a_1$ is a primitive $k$-root of unity so 
    that $g_\gamma^k=z+z^n+cz^{2n-1}$ with $b_1=\frac{1}{ka_1}, \, c_1=\frac{ca_1}{k}-\frac{n(k-1)}{2k^2a_1^3}$ 
    where $n-1$ is a multiple of $k$.
\end{enumerate}
\end{proposition}

The conjugacy map $\gamma$ in the theorem above may be not  unique. 
According to the classical results of B\"ottcher, Cremer, K\"onig, Fatou,  Leau,
Schr\"{o}der, and Siegel; if $g$ has positive radius of convergence, 
the map $\gamma\in \Gamma$ can be choosen also with non-zero 
radius of convergence in all cases except when 
$a_1=\exp(2\pi i \alpha)$ where $\alpha$ belongs to the class of 
so-called  Cremer  non-linearizable examples. 

We assume that the reader is familiar with the basic notions in 
holomorphic dynamics such as the Julia, the filled Julia  and the Fatou sets 
for semigroups of holomorphic endomorphisms of either the complex plane $\C$
or the Riemann $\overline{\C}$ as well as the notion of the measure of maximal 
entropy for rational maps, see \cite{HinkMartSemi,L, Mdyn}. 

Now, we recall some facts about  relations and functional 
equations which appear in holomorphic dynamics. The following
proposition  summarizes the results in \cite{EreFunc, Ye} and \cite{YuanZhangCal}.

\begin{proposition}\label{pr.ErLevRitt} 
Let $P$ and $Q$ be non-exceptional rational maps, then the following statements
are true:

\begin{enumerate}
    \item The maps $P$, $Q$ share the measure of maximal entropy 
    if and only  if $P$ and $Q$ share the sets of preperiodic points or, equivalently,  
    there are iterations $P^n,Q^m$ satisfying  $$P^n \circ Q^m =P^{2n} $$ and $$Q^m\circ P^n=Q^{2m}.$$
 \item If $P$ and $Q$  commute then the abelian semigroup 
 $\langle P, Q \rangle$  is not free abelian, that is, $P$ and $Q$ share an iteration. 
\end{enumerate}
\end{proposition}
A semigroup satisfying the relations in part 1 of last proposition is known as 
left weakly archimedian. However, in dynamics of rational maps these relations appeared 
in the work of Levin \cite{Levinrelations}. So, we sometimes will refer to them as \textit{Levin relations}.

\section{Proofs of theorems}

In this section, we give the proof of the theorems which were 
stated in the introduction. 

\subsection{Semigroups of entire maps and formal series without 
non-cyclic free subsemigroups}

To prove Theorem \ref{th.polysemi}, we need some basic notions of Ore's theory on 
reversible cancellative semigroups, see \cite{PatersonBook}. A semigroup $S$ is called 
\emph{left (right) reversible} or that satisfies \emph{left (right) 
Ore's conditions} if for every pair $$a,b\in S$$ there exists a 
pair $x,y\in S$ such that $xa=yb$ ($ax=by$). In other words, every 
pair of left (right) principal ideals have non-empty intersection. 
It is known, that a cancellative 
left (right) reversible semigroup embeds into a group 
$\mathcal{F}(S)=S^{-1}S$ ($\mathcal{F}(S)=SS^{-1}$) which is 
called the group of left (right) fractions of $S$, see \cite{Donnelly, PatersonBook}.

The group  $\mathcal{F}(S)$ is unique in the following sense: 
If $S\subset G$ is a left (right) reversible subsemigroup of the group $G$, generating 
$G$ as a group, then $G$ is isomorphic 
to $\mathcal{F}(S)$.

Even more, by construction (for details, see \cite{GranirerExtremely}), 
the group of right fractions of a cancellative semigroup $S$ is amenable if 
and only if $S$ is left amenable. By using the opposite semigroup argument
we can interchange left by right properties. 

Now, we are ready to prove Theorem \ref{th.polysemi}. 

\begin{proof}[Proof of Theorem \ref{th.polysemi}]
Let $S\subset G$ be a the collection of all rational elements
in $G$. Then, $S$ is a semigroup which generates $G$ as 
a group. Assume, $S$ contains a non-exceptional rational map, 
then by hypothesis and  Theorem \ref{th.AmenabilityI}, $S$ is a right-amenable 
semigroup of rational maps. Thus $S$ is right reversible. Since 
$S\subset G$ is cancellative, then bysee
Ore theory on reversible cancellative semigroups, $S$ embeds into 
the amenable group $\mathcal{F}(S)$ of its right fractions, 
which is isomorphic to $G$ by the unicity of $\mathcal{F}(S).$ 
This finishes the non-exceptional case. 

Now, assume that $S$ consists only of exceptional rational maps.  
By assumptions and by Proposition \ref{pr.Bottcher}, $S$ consists 
either of M\"obius maps or of rational maps which are 
conjugated to either a power polynomial, a Tchebichev  polynomial 
or a Latt\`es map. 

If all elements in $S$ have degree 
$1$, then $S$ is subsemigroup of the group 
$\mathrm{Mob}(\C)$, of M\"obius transformations in $\C$ fixing 
the origin. Hence, by conjugacy, the group $G$ is isomorphic 
to a subgroup of the amenable group $\mathrm{Aff}(\C)$. Therefore, $G$ is amenable. 
 
Now, let $p\in S$  be a non-injective rational map with  minimal degree 
$d\geq 2$ and let $\gamma \in \mathrm{Mob}(\C)$  be the 
conjugating map  of $p$ to either a monomial, a Tchebichev
or a Latt\`es map. 

We work out the monomial case, while the other cases follow 
by analogous arguments. We can assume that $\gamma \circ p \circ \gamma^{-1}(z)=z^d$ 
and $\gamma(0)=1$. Since every element of $S$ is exceptional, then by assumptions and
by Corollary 3.3 in \cite{Tucker}, the semigroup 
$\tilde{S}=\gamma \circ S \circ \gamma^{-1}$ consists only of monomials. 

First, we claim that every non-injective element $q$ in 
$\tilde{S}$ is a power map. 
Indeed, by assumptions and by Proposition \ref{pr.Bottcher}, 
every element $q\in \tilde{S}$ 
fixes $1$ and has the form $\omega z^k$, hence $\omega=1$ as claimed. 

The only element of degree $1$ in $\tilde{S}$ is the identity. 
Otherwise, if $\tau\neq Id\in \tilde{S}$ is an element of degree $1$, then 
$\tau\circ z^d$ is a non-injective element which is not a power 
map contradicting the claim.   Then, $S$ is an abelian semigroup 
generating an abelian group $G$. Thus $G$ is amenable. \end{proof}

Now, we prove Corollary \ref{cor.titsalternative}.
\begin{proof}[Proof of Corollary \ref{cor.titsalternative}]
Let $S$ be the semigroup of all rational elements of $G$, 
then as follows from the proof of Theorem \ref{th.polysemi},
$S$ is either abelian or an amenable cancellative semigroup of 
non-exceptional rational maps. By Theorem 1 in \cite{CMAmenability}, 
the group of right fractions of $S$, which is isomorphic to $G$, 
is virtually cyclic. 
 \end{proof}

To prove Theorem \ref{th.superatracting}, we need the following lemmas.

\begin{lemma}\label{lem.forth6}
Let $S\subset \mathfrak{S}_2$ be a semigroup, then either
\begin{itemize}
\item $S$ contains a non-cyclic free subsemigroup,
\item $S$ contains a non-cyclic free abelian semigroup, or 
\item for every pair of elements $f,g\in S$ there exist numbers $k,l$ such that
$$f^k\circ g^l=f^{2k}$$ and  $$g^l\circ f^k =g^{2l}.$$
\end{itemize}
\end{lemma}

\begin{proof}
Let $f(z)=z^m
a_{m+1}z^{m+1}+...$ and $g(z)=z^k
b_{k+1}z^{k+1}+...$, $k,m\geq 2$, be arbitrary elements of $S$. If 
$\gamma_f \in \Gamma_1$ is the ``B\"ottcher coordinate" of $f$,
then by Proposition \ref{pr.Bottcher} either the semigroup 
$\langle f, g \rangle$ is free or 
$$\gamma_f\circ \langle f, g\rangle \circ \gamma_f^{-1}=\langle z^m,\omega z^k\rangle,$$
where $\omega$ is a root of unity. 

From here, we can assume without loss of generality, that 
$f=z^m$ and $g=\omega z^k.$ We claim that there exist natural 
numbers $\alpha, \beta\in \mathcal{N}$ such that $f^\alpha\circ g$ commutes
with $f^{\beta}.$ Indeed, let $W$ be the multiplicative group generated by $\omega$. 
Then, there are  numbers $r,s$ such that $W=Deck(f^r)\times Aut(f^s)$, where
$Deck(f^r)=\{w\in W: f^r(wz)=f^r(z)\}$ and $Aut(f^s)=\{ w\in W: f^s(wz)=wf^s(z)\}$. Now, 
setting $\alpha=rs$ and $\beta=s$ we get the numbers as claimed.

So, $\langle f^\alpha \circ g,f^\beta\rangle$ is either a free abelian semigroup 
or $f^\alpha \circ g$  shares an iteration with $f$. Hence, there are numbers
$d, \tilde{m},\tilde{k}$ so that $m=d^{\tilde{m}}$ and $k=d^{\tilde{k}}$. As above, 
we have $\omega=\lambda \tau$ with $\lambda\in Aut(g^{\tilde{s})}$ and $\tau\in Deck(g^{\tilde{r}})$
for suitable $\tilde{r}$ and $\tilde{s}$. Thus, there exists an iteration $n$ so that  $g^n=\tilde{\tau}z^{k^n}$
with $\tilde{\tau}\in Deck(g^n)$.
If $\gamma$ and $\delta$ are natural numbers such that $m^\gamma=k^\delta$ then 
$g^{n\gamma}\circ f^{n\delta}=g^{2n\gamma}$. By construction, 
$Deck(g^{n\gamma})\subset Deck(f^{n\delta})$ so we also have 
$f^{n\delta}\circ g^{n\gamma}=f^{2n\delta}.$\end{proof}

The following lemma serves as an analogue of the implication from part 2 to part 3 of Theorem \ref{th.Several}.

\begin{lemma}\label{lm.preperiodic} If $f,g\in Ent(\C)$, $l$ and $k$ are integers such that 
$$f^n\circ g^k=f^{2n}, $$ and $$ g^k\circ f^n=g^{2k},$$
then $f$ and $g$ share the set of preperiodic points. 
\end{lemma}

\begin{proof}
By symmetry on $f$ and $g$, it is enough to show that the set of preperiodic points of $g$ belongs 
to the set of preperiodic points of $f$. From the first equality $g^k$ is 
semiconjugated to $f^n$ by $f^n$. Therefore, if $x$ is preperiodic for $g^k$ then $f^n(x)$ is 
preperiodic for $f^n$, thus $x$ is preperiodic for $f$. 
\end{proof}

We proceed to the proof of Theorem \ref{th.superatracting}.

\begin{proof}[Proof of Theorem \ref{th.superatracting}] 
Let $f$ and $g$ be entire maps and $z_0$ be a superattracting 
periodic point for $g$  as in the conditions of the theorem. By an appropriate conjugacy,
we can take $z_0=0$. Let $k$ and $l$ be the periods of $0$ for $f$ and $g$, 
respectively, then  $\langle f^k,g^l\rangle \subset \mathfrak{S}$. If $0$ is 
superattracting for $f$, then Lemmas \ref{lem.forth6} and \ref{lm.preperiodic} finish
the proof. Otherwise, $f^k(z)=\lambda z+...$ with $\lambda \neq 0$ 
and $f^k\circ g^l=\lambda z^n+...$.
If $\gamma_g$ is the B\"ottcher coordinate for $g^l$, then by Proposition \ref{pr.Bottcher} either 
$\gamma_g\circ \langle f^k\circ g^l,g^l\rangle\circ \gamma_g^{-1}$ is free or 
$\gamma_g\circ \langle f^k\circ g^l,g^l\rangle\circ \gamma_g^{-1}=\langle 
wz^n,z^n\rangle$ where $w$ is a root of unity. So, 
$$\gamma_g\circ [f^k\circ g^l]\circ \gamma_g^{-1}=\gamma_g\circ f^k\circ 
\gamma_g^{-1}\circ z^n =wz^n.$$ Hence, by right cancellativity $\gamma_g\circ f^k\circ \gamma_g^{-1}(z)=wz$. 
As $f$ is a non-injective entire map and $w$ is a root of unity, the last equation contradicts the
assumption $\lambda\neq 0.$
\end{proof}

Now, we show Theorem \ref{th.sharingiteration}.

\begin{proof}[Proof of Theorem \ref{th.sharingiteration}]

The following implications are straightforward (3) implies (2) implies (1). Hence
it is enough to show that (1) implies (3). We have to show that $f$ and $g$ commutes.
Since $f$ is of infinite order, by Proposition \ref{pr.SeriesLineariz}, 
there exists $\gamma_f\in \Gamma$ such  that for $f=\lambda z + ...$ we have 
$\gamma_f\circ f \circ \gamma^{-1}_f(z)=\lambda z$. Since $g$ shares an 
iteration with $f$, then an iterate of $g$ is also conjugated to a monomial and,  hence
$\gamma_f \circ g \circ \gamma_f^{-1}=tz$ for some $t\in \C$, which implies that $f$ and $g$ commutes as we wanted to show.

If $\lambda$ is a root of unit and $f(z)=\lambda z+a_nz^n+...$ where 
$a_n$ is the first non-zero coefficient after $\lambda z$. By item 3 of
Proposition \ref{pr.SeriesLineariz}, we have 
$\gamma_f\circ f\circ\gamma_f^{-1}(z)=\lambda z+bz^n+cz^{2n-1}$ for suitable constants $b$ and $c$. Thus, we have  $$\gamma_f\circ \langle f, g\rangle\circ \gamma_f^{-1}
=\langle \lambda z +b z^n+ c z^{2n-1}, \gamma_f\circ g \circ\gamma_f^{-1}\rangle,$$ 
where $$q(z):=\gamma_f\circ g \circ \gamma_f^{-1}(z)=wz+ b_m z^m+...$$ 
is a polynomial with $w$ a root of unity.

If  $p(z):=\gamma_f \circ f \circ \gamma_f^{-1}(z)$, then by 
Theorem \ref{th.Several}, Corollary 24 in \cite{CMAmenability} and 
Theorem 1 in \cite{AtelaHu}, there exist natural numbers $l$ 
and a root of unity $\tau$ such that $q=\tau p^l$, and the 
map $z\mapsto \tau z$ commutes with $p^l$ and thus with $p$. 
Thence, $q$ commutes with $p$ and we are done. 
\end{proof}

 Now we continue to prove Theorem \ref{th.preperiodic}.

\begin{proof}[Proof of Theorem \ref{th.preperiodic}]
Let $f, g$ and $D$ be as in the hypothesis of the theorem. Let 
$p$ and $q$ be the periods satisfying $f^p(D)=D=g^q(D).$ 
Let $h:\mathbb{D}\rightarrow D$ be a Riemann map, then the semigroup 
$\hat{S}=h^{-1}\circ \langle f^p, g^q \rangle \circ h\subset Rat(\C)$ and 
is a semigroup generated by two finite degree Blaschke maps 
$B_f=h^{-1}\circ f^p\circ h^{-1}$  and $B_g=h^{-1}\circ g^q\circ h^{-1}$.   

In order to avoid intersection with  Theorem \ref{th.superatracting}, 
in the proof of Theorem  \ref{th.preperiodic} we can assume that 
either $B_f$ and $B_g$ do not have a common fixed point in 
$\mathbb{D}$ or the common fixed point is not superattracting  for none of $B_f$ and $B_g.$ 

First, assume that $\hat{S}$ contains a non-exceptional element, say $B$, then by 
Theorem \ref{th.AmenabilityI}, every element $s$ of $\hat{S}$ shares the measure of maximal 
entropy with $B$. By Proposition \ref{pr.ErLevRitt}, then $B$ and $s$ satisfy the Levin relations, 
in particular, by Lemma \ref{lm.preperiodic} $f$ and $g$ share the set of preperiodic points.
Hence, we can assume that $\tilde{S}$ consists only of exceptional non-constant rational maps. 

Let $B_f=\lambda z$ with $|\lambda|=1$ be an irrational rotation (thus $D$
is a Siegel disk for $f$) and $B_g$ is conjugated to either $z^{deg(B_g)}$ if
$deg(B_g)>1$ or to an irrational rotation when $deg(B_g)=1.$ We claim that 
$deg(B_g)=1$ and $B_g$ is a rotation around $0.$ Indeed, if $deg(B_g)>1$, 
then the maps $B_g$ and $B_f\circ B_g$ define a right amenable semigroup 
of rational maps and hence these maps share the measure of maximal entropy. Thus, 
the measure of maximal entropy $\mu_g$ of $B_g$ is invariant for $B_f$. 
Then $\mu_g$ is the one-dimensional Lebesgue measure of the unit circle, since
$B_f$ is an irrational rotation. Therefore, $B_g=w z^{deg(B_g)}$ with $|w|=1$.
It follows that the semigroup $\langle \lambda z, wz^{deg(B_g)}\rangle$ is free by 
Proposition \ref{pr.Bottcher}, this is a contradiction.

Hence $deg(B_g)=1$ and by the classification of periodic bounded Fatou 
components for entire maps, $B_g$ is an irrational rotation around 
fixed point in $\mathbb{D}$. 
With the same token we get that $D$ is a Siegel disk for $g$.
Hence $\hat{S}\subset Aut(\mathbb{D})\cong PSL(2,\mathbb{R})$.

To complete the proof of the theorem, we need the following Lemma, 
which can be regarded as a version of Tits alternative for  finitely 
generated non-discrete subsemigroups of $Aut(\mathbb{D})$.

\begin{lemma}\label{lm.TitsD}
Let $S=\langle \gamma, \tau \rangle\in Aut(\mathbb{D})$ be a semigroup generated by 
the irrational rotations $\gamma$ and $\tau$ with centers $c_\gamma$ and $c_\tau$, respectively. Then either 
$c_\gamma=c_\tau$  so $S$ is abelian or $c_\gamma\neq c_\tau$ and
$S$ contains a non-cyclic  free subsemigroup. 
\end{lemma}

\begin{proof}
If $c_\gamma=c_\tau$, then $S$ is abelian. We can assume that $c_\gamma=0$ 
and $\gamma(z)=\lambda z$ and $c:=c_\tau\neq 0$. Let $\sigma$ be a 
hyperbolic geodesic in $\mathbb{D}$ passing by $0$ and $c$.  Let $\triangle_0$ 
and $\triangle_c$ be the sets of angles of all the rotations in the semigroups 
$\langle \gamma \rangle$ and $\langle \tau \rangle$, respectively. Let $\alpha$ and  $\beta$ 
be disjoint geodesics passing through $0$ and $c$, respectively, so that the angle  
between $\sigma $ and $\alpha$ belongs to $\triangle_0$ and the angle between 
$\sigma$ and $\beta$ belongs to $\triangle_c$.  We can construct such configuration 
of geodesics since $\gamma$ and $\tau$ are irrational rotations. 

Let $r_\sigma,r_\alpha,r_\beta$ be the reflections with respect to
$\sigma,\alpha$ and $\beta$ respectively, then the element 
$h=r_\beta\circ r_\alpha$ is  hyperbolic and 
$$h=r_\beta \circ r_\sigma \circ r_\sigma \circ r_\alpha=[r_\beta\circ r_\sigma]\circ [r_\sigma \circ r_\alpha].$$
By construction, $$r_\beta \circ r_\sigma \in \langle \gamma \rangle$$
and $$r_\sigma\circ r_\alpha\in \langle \tau \rangle.$$
Hence, there are numbers $k$ and $l$ such that $$h=\gamma^k \circ \tau^l.$$

Now, let $U\subset \mathbb{D}$ be a fundamental domain for the action of $h\in S$ 
on $\mathbb{D}.$ If $S$ contains another hyperbolic element $g$ with fundamental
domain $V$ such that  $\overline{\C}\setminus \mathring{U}\subset V$ and
$\overline{\C}\setminus \mathring{V}\subset U$, where $\mathring{V}$ denotes 
the interior of $V$. Then by the Klein combination theorem (see for example \cite{KapovichHyp}), 
the semigroup  $\langle h, g \rangle $ is free as it is a subsemigroup of a Schottky non-cyclic 
free group generated by $h$ and $g$.  The 
existence of $g$ above follows from the elementary fact in hyperbolic geometry: 
\textit{The semigroup $S$ contains a dense subset 
of rotations around almost every point in the unit disk $\mathbb{D}.$} 

We finish the lemma with the sketch of the proof of this fact. Let $z$ be a point in 
$\mathbb{D}$ so that $(z,0,c)$ forms a  non-degenerated hyperbolic 
triangle in $\mathbb{D}$. Then a suitable  rotation around $z$ 
can be expressed as a composition of rotations  around $0$ and $x$ (using suitable 
reflections inside the triangle, as above).
Since $\gamma$ and $\tau$ are irrational rotations in $S$ around $0$ and $c$, respectively, 
and all angles in the triangle $(z,0,c)$ depends continuously,  every 
neighborhood of $z$ contains a point $z_0$ with a dense set of rotations 
around $z_0$ in $S$. \end{proof}
We continue the proof of Theorem \ref{th.preperiodic}.

By Lemma \ref{lm.TitsD} and by assumptions, the semigroup $\hat{S}$ (and hence $S$) is abelian 
and there are integer $m,n\geq 1$ such that $B_f^m=B_g^n$ (and hence $f^{pm}=g^{qn}$). 

Finally, by assumptions, Theorem \ref{th.AmenabilityI}, Corollary 3.3 in \cite{Tucker} (see also 
\cite{CMAmenabilityII}) and by Theorem 4 in \cite{CMAmenabilityII}, we may assume that  $B_f(z)=z^k$ and 
$B_g(z)=\omega z^m$ for some natural numbers $k$ and $m$ and some root of unity $\omega.$
 Now, we can apply Lemmas \ref{lem.forth6} and \ref{lm.preperiodic} to finish the proof.
 
 \end{proof}

\begin{proof}[Proof of Corollary \ref{cor.boundeddomain}] 

Let $h:D\rightarrow \mathbb{D}$ be a Riemann map, then the semigroup
$\hat{S}=h\circ S \circ h^{-1}$ is non-exceptional. Thus by Theorem \ref{th.AmenabilityI}, 
every non-injective element of $\tilde{S}$ shares the  measure of 
maximal entropy with  $B_f=h\circ f\circ h^{-1}.$ Since $\tilde{S}$ does 
not contain injective elements, by Proposition \ref{pr.ErLevRitt} every pair 
of elements satisfy Levin relations. Thus by Lemma \ref{lm.preperiodic}, 
$f$ and $g$ share the set of preperiodic points. 
\end{proof}

\begin{proof}[Proof of Corollary \ref{cor.SiegelAbel}]
As in Lemma \ref{lm.TitsD} let $h:D\rightarrow \mathbb{D}$ be a Riemann map. 
Hence $\hat{S}=h\circ \langle f,g\rangle \circ h^{-1}$ is a semigroup of irrational rotations. 
Let $B_f=h\circ f \circ h^{-1}$ and $B_g=h\circ g \circ h^{-1}$. 
Then either the centers coincide and $S$ is abelian or $\tilde{S}$ contains a hyperbolic element $\gamma=B_f^s\circ B_g^r$ for 
suitable numbers $s,r$. Then $D$ is an invariant Fatou component for $f^s\circ g^r$ without fixed points and injective on $D$
which contradicts the classification of bounded invariant Fatou connected components of entire maps. 
\end{proof}

\begin{proof}[Proof of Theorem \ref{th.boundedFatou}] Let $h:D\rightarrow \mathbb{D}$ be 
a Riemann map.  Then 
$B_f=h\circ f \circ h^{-1}$ is a  Blashke endomorphism of 
the unit disk. Also  $B_f=h\circ f \circ h^{-1}$ is a 
non-injective Blashke endomorphisms since $B_f$ shares the set of 
preperiodic points with $B_g$.
By Theorem 1.1.5 in \cite{Ye}, $B_f$ and $B_g$ shares the measure 
of maximal entropy. By Theorem \ref{th.AmenabilityI}, we get that $h\circ \langle f, g \rangle
\circ h^{-1}$ is a non-exceptional right amenable semigroup. We apply
Theorem \ref{th.AmenabilityI} and Proposition \ref{pr.ErLevRitt} to
get the conclusion of the theorem.   
\end{proof}

\subsection{Right amenable and nearly abelian semigroups of entire maps}
 
The idea motivating the following Lemma appears in \cite{HinkMartSemi} 
in the context of rational maps. Indeed, the proof of the lemma does not 
require additional arguments, but we include them for the reader's 
convenience. 

\begin{lemma}\label{lm.algebraic} Let $S$ be a nearly abelian 
semigroup of entire maps and let $G(S)$ be the group generator by 
the commutator set  $K(S)$. Then every element  $s\in S$ acts on $G(S)$ 
by  semiconjugacy as an endomorphism, that is, for every $\gamma \in G(S)$ 
there exists $\gamma_s\in G$ such that $$s\circ \gamma=\gamma_s\circ s.$$
\end{lemma}

\begin{proof}
    Take $s\in S$ and $\omega \in K(S)$, then there exists a pair $f,g\in S$ such
    that $f\circ g=\omega \circ g \circ f$ and 
    
    $$s\circ \omega \circ g \circ f= s\circ f\circ g= \tau \circ f \circ [s\circ g]=\tau \circ h \circ [s\circ g] \circ f,$$ for
    some $\tau, h \in K(S)$, hence $$s\circ \omega=\tau \circ h \circ s=\omega_s\circ s.$$  Thence, for every $\gamma \in G(S)$ there exists $\gamma_s\in G(S)$ such that
    $$s\circ \gamma=\gamma_s\circ s$$ as we wanted to show. 
    
\end{proof}

\begin{proof}[Proof of Theorem \ref{th.nearlyabelian}] 
By Lemma \ref{lemma.leftright},  
we can assume  that $S$ is countably generated. Let $S=\langle f_1,..., 
f_n,..\rangle$ be a countably generated nearly abelian 
semigroup, such that the set of $\{f_i\}$ forms a minimal  generating 
set. If $S_n=\langle f_1,...,f_n\rangle$, then $\{S_n\}$ 
forms a nested sequence of finitely generated semigroups with
$S=\bigcup S_n$. Hence, the semigroups  $\langle G(S), S_n\rangle$ are 
nearly abelian by Lemma \ref{lm.algebraic} and  also  these form a 
nested sequence of subsemigroups of $\langle G(S),S\rangle$ with $\langle G(S),S\rangle =\bigcup_{n\geq 0} \langle G(S), S_n\rangle$.
By Day's exhaustion theorem in Section 2.1, it is enough to show that $\langle G(S),S_n\rangle$ is right amenable.
Indeed, first notice that $G(S)\subset \mathrm{Aff}(\C)$ is right amenable as 
a subgroup of an amenable group. By Lemma \ref{lm.algebraic}, 
the semigroup $\langle G(S), f_1 \rangle$ is the semidirect product of 
the group $G$ and the cyclic semigroup $\langle f_1 \rangle$, hence is 
right amenable by Proposition \ref{LemmaKlawe}.

Now, we proceed by induction, if $\langle G(S),S_n\rangle$ is right 
amenable, then again by assumption and Lemma \ref{lm.algebraic} 
$f_{n+1}$ acts on $\langle G(S),S_n\rangle$ 
making $\langle G(S), S_{n+1} \rangle$ a semidirect product 
of right amenable subsemigroups $\langle G(S),S_n\rangle$ and $\langle f_{n+1} \rangle$,
we can apply again Proposition \ref{LemmaKlawe} to finish the proof. 
\end{proof}

\subsection{Left amenability and Stone-C\v{e}ch extension}

To prove Theorem \ref{th.AMEStoneCech}, we need the following proposition.
\begin{proposition}\label{pr.TheoremStoneCech}
Let $X$ be a compact Hausdorff space and let $S$ be a left-amenable 
semigroup of non-constant continuous endomorphisms of $X$, then $S$ 
admits an invariant probability Borel measure on $X$.
\end{proposition}

\begin{proof}
Let $\sigma$ be any probability Borel measure on $X$ and 
$H:C(X)\rightarrow L_\infty(S)$ is a 
continuous linear map given by 
$$H(\phi)(g)=\int_X \phi(g(t))d\sigma(t),$$
for all $\phi\in C(X)$ and $g\in S.$ If $Y=\overline{H(C(X))}$, then 
$Y$ contains $\chi_S$, the characteristic function of $S$, since $\chi_S=H(\chi_X)$ where 
$\chi_X$ is the characteristic function of the space $X$.
Hence, if  $L\in LIM(S)$, then the functional $$\ell(\phi)=L(H(\phi)), \phi\in C(X)$$ 
is continuous, linear, positive and non-zero. We claim that for 
every $h\in S$ and every $\phi \in C(X)$ we have $$\ell(\phi)=\ell(\phi(h)).$$ 
Indeed, $$\ell(\phi(h))=L(H(\phi(h)))=
L\left(\int_X [\phi(h\circ g)](t)d\sigma(t)\right)$$

$$=L(l_h(H(\phi)))=L(H(\phi))=\ell(\phi).$$
By Riesz representation theorem, $\ell$ is represented by an invariant 
probability measure $\mu.$ 
\end{proof}

\begin{proof}[Proof of Theorem \ref{th.AMEStoneCech}]
By the property of universality of the Stone-\v{C}ech compactification $\beta(\C)$, 
the whole semigroup $Ent(\C)$ extends to $\beta(\C)$ as a subsemigroup of $E(\beta(\C)).$ 
Thus, $S$ also extends to a  subsemigroup, say $\tilde{S}$, of  $E(\beta(\C))$.  Then 
$\tilde{S}$ is a left-amenable semigroup of non-constant continuous endomorphisms of $\beta(\C)$ and, 
by Proposition \ref{pr.TheoremStoneCech}, $\tilde{S}$  admits an invariant 
probability measure on $\beta(\C)$.     
\end{proof}

\subsection{Ruelle representation}

We start the subsection with the following definitions. 

\begin{definition}

Let $S\subset Ent(\C)$ be a semigroup of entire maps, we will say  that $S$ is \emph{quasiconformally 
deformable}, whenever there exists a quasiconformal automorphism $h$ of $\C$ such that 
$$h\circ S \circ h^{-1}\in Ent(\C)$$ and there is no $\gamma\in \mathrm{Aff}(\C)$ with $$h\circ S \circ h^{-1}=\gamma \circ S \circ \gamma^{-1}.$$ 
\end{definition}
We will say that the map $f$ is \emph{quasiconformally deformable}, whenever the cyclic semigroup 
$\langle f \rangle$ is quasiconformally deformable.

\begin{definition}
Let $\mathcal{F}\subset  Ent(\C)$ be an analytic family. We say 
that a semigroup $S=\langle f_1,f_2,...\rangle \subset Ent(\C)$ is 
\emph{stable over} $\mathcal{F}$ if   $f_i\in \mathcal{F}$, for all $i$, and for every 
monomorphism (representation) $\phi:S\rightarrow Ent(\C)$ 
close enough to the identity on the generators
of $S$, in the compact-open topology, with $\phi(f_i)
\in \mathcal{F}$ there exists a quasiconformal automorphism $h$ of $\C$ such that 
$$\phi(g)=h\circ g \circ h^{-1},$$ for every $g\in S.$ 

\end{definition}
We say that $f$ is \emph{stable in} $\mathcal{F}$, whenever $\langle f\rangle$ 
is stable over $\mathcal{F}$.
\begin{definition}

We call a countably generated semigroup $S\in Ent(\C)$ a \emph{$D$-semi\-group} 
 if $S=\langle f_1,f_2,... \rangle$ with
$f_i\circ f_j=f_i^2$, for all $i$ and $j$. 

\end{definition}

\begin{proposition}\label{pr.Estrella}
A $D$-semigroup $S$ is quasiconformally deformable, whenever 
a generator of $S$ is quasiconformally deformable.
\end{proposition}

\begin{proof}
Assume that the generator $f_1\in S$ is a quasiconformally deformable 
with associated quasiconformal automorphism $h$. Let 
$\mu=\frac{h_{\overline{z}}}{h_z}$ be the Beltrami coefficient of $h$.
If $\mu(f_i)\frac{\overline{f'_i}}{f'_i}=\mu$
for every generator $f_i$, then $h\circ S \circ h^{-1}\subset  Ent(\C)$. Indeed, 
this follows from the relations on the generators and the 
$f_1$-invariance of $\mu$. \end{proof}

In contrast with the setting of $D$-semigroups, we have the following theorem. 

\begin{theorem}\label{th.qcdeformable}
Let $P$ and $Q$ be two polynomials generating a free semigroup 
$S=\langle P, Q \rangle$ then $S$ is not quasiconformally deformable.
\end{theorem}

\begin{proof}
First let us assume one of the polynomials, say $P$, is 
non-exceptional and proceed by contradiction. Suppose that $S=\langle P,Q \rangle$ 
is quasiconformally deformable and $\mu$ is a non-zero $S$-invariant 
Beltrami differential. We claim that $\mu$
is non-zero on the basins of attraction of infinity $A_\infty(P)$ and $A_\infty(Q)$. 
Indeed, assume $\mu$ is zero in $A_\infty(P)$, the 
basin of attraction of infinity of $P$. By assumptions of the theorem, 
Theorem \ref{th.AmenabilityI} and Theorem \ref{th.Several}, 
we can assume that $J(P) \cap J(Q)\neq J(P)$. If  $\mu=0$ on $A_\infty(P)$, 
then $J(P)\subset K(Q)$ the filled Julia set of $Q$, and $K(P)\subset K(Q)$.  
But $\mu$  is also invariant for $Q$, therefore $\mu=0$ in $A_\infty(Q)$ 
and $K(Q)\subset K(P)$ which contradicts that $J(P)\cap J(Q)\neq J(P).$
Hence, $J(P)=J(Q)$ and, by Theorem \ref{th.Several},  $S$ is 
right amenable which is a contradiction. So $\mu$ is 
non-zero  on $A_\infty(P)$. 

The restriction $\nu=\mu|_{A_\infty(P)}$ 
is determined by  the equipotential foliation since $\mu$ is 
constant on almost every  leaf of this foliation. The set $A_\infty(P)$ 
is saturated by the equipotential foliation and 
$\mu$ is invariant for $Q$, then $A_\infty(P)$ is completely 
invariant for $Q$. The harmonic invariant measure $m_P$ 
for $P$ in $A_\infty(P)$ with respect to infinity 
is a harmonic invariant measure for $Q$. By Brolin theorem $P$ 
and $Q$ share the measure of maximal entropy, by 
Theorem \ref{th.AmenabilityI} $S$ is right amenable which 
contradicts the hypothesis.

Finally, let us consider the case where both $P$ and $Q$ are 
exceptional. Without loss of generality we can assume that 
$P$ is $z^n$ or a Tchebichev polynomial. Following similar 
arguments as above, we can show that $J(Q)=J(P)$ and then 
$Q$ is of the same kind as $P$, either a monomial 
$\lambda z^n$ with $|\lambda|=1$ or a Tchebichev polynomial. 
So, if $S=\langle z^n, \lambda z^m \rangle$ and $f$ is a 
quasiconformal deformation of $S$ fixing $\{0,1,\infty\}$ 
with $f\circ z^n=z^n\circ f$,  then $f$ is the identity in the 
unit circle $\mathbb{S}^1$. Therefore, $f$ commutes with 
$\lambda z^m$ and $f\circ S \circ f^{-1}=S$ is a trivial deformation. 
If both $P$ and $Q$ are Tchebichev, we can 
apply similar arguments. 
\end{proof}

We proof Theorem \ref{th.poldeformable}.

\begin{proof}[Proof of Theorem \ref{th.poldeformable}]
The proof is an immediate consequence of Theorems \ref{th.AmenabilityI} 
and \ref{th.qcdeformable}. 
\end{proof}

For $D$-semigroups of entire maps we can show the following result.

\begin{lemma}\label{lm.LEVINA}
Every subsemigroup of a $D$-semigroup is right amenable.
\end{lemma}
\begin{proof}
Since every element of a $D$-semigroup $S$ is an iteration of a 
generator, every subsemigroup of $T\subset S$ is a $D$-semigroup. 
So, it is enough to show  that the $D$-semigroup $S=\langle f_1,f_2,... \rangle$ 
is right amenable.  By Day's exhaustion theorem it is enough 
that the finitely generated $D$-semigroup $S_n=\langle f_1,...,f_n\rangle$ is
right amenable. Indeed, this is the case by Theorem 28 in \cite{CMAmenability}.
\end{proof}

The discussion around dynamics of general semigroups of entire maps is 
yet unclear even in the right amenable case. However, 
we provide an example of a quasiconformally stable  
semigroup of entire maps over a suitable family $\mathcal{F}$.

\begin{definition}
For an entire map $f$, let 
$$Deck(f)=\{\gamma \in \mathrm{Aff}(\C): f\circ \gamma=f\}$$    
be the \emph{group of deck transformations}. For every $\gamma\in Deck(f)$ we define 
$$f_\gamma=\gamma \circ f \circ \gamma^{-1}=\gamma \circ f$$ and then
$$\Gamma(f)=\langle f_\gamma \rangle_{\gamma\in Deck(f)}$$ 
is a subsemigroup of $Ent(\C)$. We say that $Deck(f)$ is \emph{transitive} if it
acts transitively on  fibers $f^{-1}(z)$, for almost every point $z\in \C$.
\end{definition}

For a family $B\subset Deck(f)$, we let $\Gamma_B=\{f_\gamma|\gamma\in B\}.$
In contrast with rational maps, the group $Deck(f)$ may be  infinite. Moreover, 
typically $\Gamma(f)$ is an infinitely generated subsemigroup of $Ent(\C)$. 

\begin{proposition}\label{prop.tresestrellas}
Let $\mathcal{F}\subset Ent(\C)$ be an analytic family, and $f\in \mathcal{F}$. Assume that 
$Deck(f)$ is transitive.  Then  $\Gamma_B\subset \Gamma(f)$ is  stable over $\mathcal{F}$, 
for every finite $B\subset Deck(f)$ containing $\tilde{\gamma}$ with $f_{\tilde{\gamma}}$
stable over $\mathcal{F}$.
\end{proposition}

\begin{proof} If $B\subset Deck(f)$ and $\tilde{\gamma}\in B$, then 
 $\Gamma_B$ is a finitely generated subsemigroup of $\Gamma(f)$. Let 
$\phi:\Gamma_B\rightarrow \Gamma'\subset Ent(\C)$ be an isomorphism such that
$$\phi(f_\gamma)\in \mathcal{F}$$ for every $\gamma \in B$ and $\phi(f_\gamma)$ 
is a sufficiently closed to $f_\gamma$ on compact subsets of $\C$.

By assumption, the map $f_{\tilde{\gamma}}$ is stable over $\mathcal{F}$. Let $h$ 
be a quasiconformal homeomorphism conjugating  
$\phi(f_{\tilde{\gamma}})=h\circ f_{\tilde{\gamma}} \circ h^{-1}$.  
We claim that $\phi(g)=h\circ g \circ h^{-1}$, for every $g\in \Gamma_B$.
Indeed, by assumption $$Deck(\phi(f_{\tilde{\gamma}}))=h\circ
Deck(f_{\tilde{\gamma}})\circ h^{-1}= h\circ Deck(f)\circ h^{-1}$$ which 
implies that $Deck(\phi(f_{\tilde{\gamma}}))$ is transitive, and since 
every element of $\Gamma_B$ is an iterate of a generator, 
we are done. \end{proof}

Let $f$ be an entire map in class $\mathcal{R}$ and 
let $R$ be a rectifiable ray containing all critical  and 
asymptotic values of $f$. Let $U=\C\setminus R$ 
and $V=f^{-1}(U)$, then $V=\bigcup_{i\in \mathbb{N}} V_i$ 
such that $V_i$ are the simply connected components of $V$ 
and the restrictions $f:V_i\rightarrow U$ are holomorphic 
homeomorphisms. Set $f_i=f|_{V_i}$. For every bijection 
of the natural numbers $b\in Bij(\mathbb{N})$, 
let us define a piecewise conformal automorphism of $V$ by 
$h_{b}(z)= f^{-1}_{b(i)}\circ f_i(z),$ for  $z\in V_i.$ Then $h_{b}$ 
is a measurable bijection conformal almost everywhere on $\C$.
Denote by $FD(f)$ the group generated by $h_{b}$, $b\in Bij(\mathbb{N})$. 
Hence $FD(f)$ is isomorphic to $Bij(\mathbb{N})$ and it can be 
regarded as the full deck group of $f$. In fact $f(\gamma)=f$,  for every $\gamma \in FD(f)$ 
almost everywhere. The group $FD(f)$ acts on $L_1(\C)$ by the push-forward operators $\gamma_*$ given 
by  $$\gamma_*:\phi \rightarrow \phi(\gamma^{-1})[(\gamma^{-1})']^2$$
or, equivalently, acts on measurable integrable quadratic differentials by
$$\gamma_*(\phi(z) dz^2)=[\phi(z)dz^2]\circ (\gamma^{-1}(z))=
\gamma_*(\phi(z))dz^2,$$ for every $\gamma\in FD(f)$.
For every subset $\Gamma\subset FD(f)$ and $\gamma \in \Gamma$ 
define $f_\gamma=\gamma\circ f \circ \gamma^{-1}=\gamma \circ f.$ 
Let $S(\Gamma)$ be the semigroup generated by 
$f_\gamma$,  $\gamma \in \Gamma$. Then, $S(\Gamma)$ consists of measurable non-singular, 
with respect to the Lebesgue measure, endomorphisms which are holomorphic 
almost everywhere on $\C$. For instance, if $\Gamma \subset Deck(f)$ then $S(\Gamma)\subset Ent(\C).$ 
We construct 
the Ruelle representation $\rho$ of the semigroup $S(\Gamma)$ for 
every subset $\Gamma\subset FD(f)$ given on generators as follows:
$$\rho(f_\gamma)(\phi)=\gamma_* \circ f_*(\phi)$$ sending 
$S(\Gamma)$ into $End(L_1(A))$ for every completely invariant 
set $A\subset \C$ of positive Lebesgue measure. 
As an observation, we state the following proposition.

\begin{proposition}\label{pr.Ruellerightamenable}
For every completely invariant set $A\subset \C$ of positive 
Lebesgue measure and every subset $\Gamma\subset FD(f)$, the 
Ruelle representation $\rho:S(\Gamma)\rightarrow L_1(A)$ 
is injective. Moreover, $S(\Gamma)$ (and 
hence $\rho(S(\Gamma))$) is right amenable. 
\end{proposition}

\begin{proof}
For every pair $\gamma_1,\gamma_2\in S(\Gamma)$ we have 
$$f_{\gamma_1} \circ f_{\gamma_2}=f^2_{\gamma_1}.$$ Thus, the 
conclusion follows from the construction of Ruelle operator and 
from Lemma \ref{lm.LEVINA}. 
\end{proof}

We are ready to formulate the connection of left amenability of 
Ruelle representations with the existence of invariant Beltrami differentials. 
We follow the arguments of the proof of Proposition 39 in \cite{CMAmenability}.

\begin{proposition}\label{pr.tresestrellasveras}
Let $f$ be an entire map in class $\mathcal{R}$. Let $A$ be a completely 
invariant measurable set and $\Gamma\subset FD(f)$ be a subset, 
then $S(\Gamma)$ is $\rho$-left amenable, whenever $A$ does not 
admit an invariant Beltrami differential.
\end{proposition}

\begin{proof}
    The semigroup $S(\Gamma)$ is right amenable and thence the space 
    $X_\rho \subset L_\infty(S(\Gamma))$ possess a right invariant mean 
    $m$. Recall that $X_\rho=\mathbb{C} \times Y_\rho$ the product of 
    the space of constant functions with the space $Y_\rho.$
    If $Y_\rho \subset ker(m)$, then $m$ is also left invariant, since
    $Y_\rho$ (and $X_\rho$) is left invariant subspace.
    Otherwise, there exists $\psi \in L_1(A)$ and $\nu \in L_\infty(A)$ 
    so that $m(\varphi_{\psi, \nu})\neq 0$. Then $m(g)=m(\varphi_{g,\nu})$ 
    is  continuous and $S(\Gamma)$-invariant functional on $L_1(A).$
    Since $m(\psi)=m(\varphi_{\psi,\nu})\neq 0$, then by Riesz representation 
    theorem, there exist an element $\mu\in L_\infty(A)$
    such that $\mu(f_\gamma)\frac{\overline{f'_\gamma}}{f_\gamma}=\mu$ 
    almost everywhere. Since for every $\gamma\in \Gamma$, $\mu=\mu(\gamma)\frac{\overline{\gamma'}}{\gamma'}$ almost everywhere, $\mu$ defines an invariant Beltrami differential for 
    every $f_\gamma$ on $A$, this contradiction establishes  the fact $Y_\rho\subset ker(m)$ 
    as we wanted to show. 
\end{proof}

We need the following lemma which was proved for rational and meromorphic 
maps with finitely many critical and asymptotic values, that is, for maps in 
the so called Speisser class (see \cite{DomMakSien}, \cite{MakRuelle}, \cite{MakSienbol}).
For the convenience of reader we include a sketch of the proof of this lemma 
in our context following the lines of the works cited above.

\begin{lemma}\label{lm.invariantBeltrami}
Let $f$ be an entire map in class $\mathcal{R}$. Let $h\neq 0 \in L_1(\C)$ 
with $f_*(h)=h$, then there exists an $f$-invariant Beltrami 
differential  $\mu\in L_\infty(\C)$ such that 
$\mu=\frac{\overline{h}}{|h|}$ on the support of $h$.
\end{lemma}

\begin{proof}
Let $h\neq 0 \in L_1(\C)$ such that $f_*(h)=h$, then 
the support of $h$ satisfies  $f(supp(h))=supp(h)$ 
almost everywhere. If $\{\zeta_i\}$ is a complete set of branches
of $f^{-1}$ on $\C\setminus R$, then from the inequality $$\|h\|=\|f_*(h)\|=\int \left| 
\sum_i h(\zeta_i(z)) (\zeta'_i(z))^2 \right| |dz|^2$$
$$\leq \int \sum |h(\zeta_i(z))|\zeta'_i(z)|^2|dz|^2\leq \|h\|,$$   
it follows by Fatou lemma that almost everywhere $$\left|\sum h(\zeta'_i)(\zeta'_i)^2\right|=
\sum|h(\zeta_i)(\zeta'_i)^2|.$$ Therefore, 
for almost every  $x\in supp(h)$ and every $j$  with $\zeta_j(x)\in supp(h)$   
if  $$\alpha_j(x)=h(\zeta_i(x))(\zeta'_i(x))^2 
\textnormal{ and } \beta_j(x)=f_*(h(x))-\alpha_j(x),$$ then 
$$|\alpha_j+\beta_j|=|\alpha_i|+|\beta_j|.$$ Hence
$$\frac{\beta_j(x)}{\alpha_j(x)}=k_j(x)\geq 0,$$
almost everywhere in $supp(h)$.
Thence, for almost every $x\in supp(h)$ and every $j$ with 
$\zeta_j(x)\in supp(h)$ we have  $$\frac{f^*(h(x))}{\alpha_j(x)}=
\frac{\alpha_j(x)+\beta_j(x)}{\alpha_j(x)}=1+k_j(x)\geq 1,$$
thus $h=f_*(h)=(1+k_j)h(\zeta_j)(\zeta'_j)^2.$

If $\nu=\frac{\overline{h}}{|h|}$, then 
$$\nu(f)\frac{\overline{f'}}{f'}=\nu$$ almost everywhere 
in $supp(h).$ Now, if $\nu_0=\nu(f)\frac{\overline{f'}}{f'}-\nu$, 
then $$\mu= \nu+\sum_{i\geq 0} \nu_0(f^i)\frac{\overline{(f^i)'}}{(f^i)'}$$ 
is the desired invariant Beltrami differential.
\end{proof}

\begin{proof}[Proof of Theorem \ref{th.leftRuelle}]
The proof is by contradiction, assume that there exists a 
non-zero fixed point $h\in L_1(A)$ such that $f_*(h)=h$.   
Let $\mu\in L_\infty(A)$ be the invariant Beltrami differential constructed in Lemma
\ref{lm.invariantBeltrami}.  Then  $$\langle \mu, h \rangle =\int_\C \mu h=\int_\C |h| \neq 0.$$ 
Since $f$ acts ergodically on $A$,  if $\alpha\in L_1(A)$ and $\nu\in L_\infty(A)$ are other fixed 
points for $f_*$ and $B_f$, respectively, then $\alpha$ is a multiple of $h$ 
and $\nu$ is a multiple of $\mu$, respectively. Therefore, by 
the separation principle in Proposition \ref{pr.separationprinc}, 
the operator $f_*$ is mean-ergodic in $L_1(A)$.

Now, we claim that every finitely generated subsemigroup of 
$\rho(S(\Gamma))$ is weakly almost periodic. Indeed, the operator
$T(\phi)=\overline{\mu} f_*(\mu\phi)$ is also mean-ergodic in $L_1(A)$.  A straight 
forward computation shows that $$T(\phi)(x)=\sum_{y\in f^{-1}(x)} \frac{\phi(y)}{|f'(y)|^2}$$ 
is a positive mean-ergodic operator, with norm $\|T\|\leq 1$.  By Theorem \ref{th.KORNFELIN}, 
the  operators $T$ and hence $f_*(\phi)=\mu T(\overline{\mu} \phi)$ are weakly almost 
periodic. Therefore,  every operator $(f_{\gamma})_ *$ is weakly 
almost periodic for every $\gamma\in FD(f)$, thus for every finite family 
$\mathcal{B}\subset FD(f)$, the semigroup of $\langle (f_\gamma)_*\rangle_{\gamma 
\in \mathcal{B}}$ is weakly almost periodic. Since $\rho(S(\Gamma))$ 
consists only  of iterations of its generators, we get the claim.

Assume that a subsemigroup $S\subset S(\Gamma)$ is $\rho$-left amenable. 
We claim that there exists an non-zero element $\ell\in L^*_\infty(A)$ which is
fixed by the bidual semigroup $$\rho(S)^{**}=\{t^{**}: t\in \rho(S)\}.$$
Indeed, if $L$ is a left invariant mean on $L_\infty(S)$, then  for 
$\psi \in L_\infty(A)$, $\nu\in L_1(A)$ and every $t\in \rho(S)$, the function
$$\varphi_{\psi, \nu}(t)=\int \psi t(\nu)|dz|^2$$ belongs to $L_\infty(S).$
Thus, the functional $\ell(\psi)=L(\varphi_{\psi,\nu})$ is continuous 
on $L_\infty(A)$.  As $L$ is left invariant, then 
$\ell(t^*(\psi))=\ell(\psi)$ for every  $t\in \rho(S)$ so $\ell\in L^*_\infty(A)$ 
is fixed by $\rho(S)^{**}$ as claimed. 

Now, we use standard arguments from functional analysis 
(see, for instance, \cite{DunfordSchwartz}). The functional 
$\ell$ can be presented by a charge $\sigma_\ell$, i.e., 
a finite, finitely additive set function on the Lebesgue 
$\sigma$-algebra of measurable subsets of $A$ given by $\sigma_\ell(U)=\ell(\chi(U))$ 
where $\chi(U)$ is the characteristic function of the measurable set $U\subset A$. 
By construction $\sigma_\ell$ is null on zero Lebesgue measure 
subsets of $A.$  
 
We claim that $\sigma_\ell$ is an invariant measure absolutely 
continuous with respect to the Lebesgue measure. 
To verify this fact, it is enough to show that $\sigma_\ell$ 
is countably additive.  That is, $$\sigma_\ell(\bigcup A_i)=\sum \sigma_\ell(A_i),$$ 
for every pairwise disjoint family of measurable subsets 
$A_i\subset A.$ Since $\rho(S)$ is weakly almost 
periodic on $L_1(A)$, then for every $\epsilon>0$ and every $\alpha \in L_1(A)$ 
there exists $\delta>0$  such that  $$\int_B |t(\alpha)|\leq \epsilon,$$ 
for every $t\in \rho(S)$ and every measurable 
set $B\subset A$ with Lebesgue measure less than $\delta.$

Fix $\epsilon>0$ and $\delta$ as above.  Let $X$ be a set of finite Lebesgue measure with
$X=\bigcup A_i \subset A$, where $\{A_i\}$ is a family of pairwise disjoint measurable 
subsets of $A$. Then, for every $k>0$, we have by finite additivity 
$$\sigma_\ell(X)=\sum_{i=0}^k \sigma_\ell (A_i)+ \sigma_\ell(\bigcup_{i=k+1}^\infty A_i).$$
Let $k_0$ be such that $m(X_k)\leq \delta$ with 
$X_k=\bigcup_{i=k+1}^\infty A_i$ for every
$k\geq k_0$. Then $$|\sigma_\ell(X_k)|=|L(\chi(X_k))|
\leq |L(\varphi_{(\chi(X_k),f)})| \leq \|L\|
\|\varphi_{(\chi(X_k),f)}\|$$$$=\sup_{g\in S}|
\varphi_{(\chi(X_k),f)}(g)|=\sup_{g\in S}\int_{X_k} |\rho(g)(h)|\leq \epsilon.$$
Hence $\sigma_\ell$ is finite non-zero invariant measure absolutely continuous with 
respect to Lebesgue with density $d\sigma_\ell=\omega(z)dz^2$, $\omega\in L_1(A)$.

To finish the proof, we take a suitable conjugacy to assume that 
the semigroup $S$ contains an iterate  $f^n$ and an 
element $g=f^n_\gamma$ with  $\gamma\in FD(f)$ of infinite order. Then $\omega(z)$ is a 
scalar multiple of $h$ and hence $\gamma_*(h)=h$. Since $\gamma$ has a fundamental domain in $A$
and is of infinite order, then almost every point $z\in A$ has 
a neighborhood of zero $\sigma_\ell$-measure, that is $h(z)=0$ $\sigma_\ell$-almost every where
in $A$, which is a contradiction, this completes the proof of theorem. \end{proof}

   \bibliographystyle{amsplain}
  
\bibliography{workbib}
\Addresses
\end{document}